\documentclass[a4paper,12pt]{article}
\usepackage{amssymb}
\usepackage{amsmath}
\usepackage{url}
\usepackage{xcolor}

\newtheorem{theorem}{Theorem}

\newtheorem{definition}[theorem]{Definition}

\newtheorem{remark}[theorem]{Remark}
\newenvironment{proof}[1][Proof]{\textbf{#1.} }{\ \rule{0.5em}{0.5em}}

\begin{document}

\title{Second-order nonstandard finite difference schemes for a class of models in bioscience}
\author{Roumen Anguelov$^1$ and Jean M.-S.  Lubuma$^2$\\
$^1$ Department of Mathematics and Applied Mathematics\\
University of Pretoria, Pretoria, South Africa\\
Email: roumen.anguelov@up.ac.za\\
$^2$School of Computer Science and Applied Mathematics\\
University of the Witwatersrand, Johannesburg, South Africa\\
Email: jean.lubuma@wits.ac.za}
\date{}
\maketitle

\begin{abstract}
\noindent We consider a dynamical system, defined by a system of autonomous differential equations, on $\Omega\subset\mathbb{R}^n$. By using Mickens' rule on the nonlocal approximation of nonlinear terms, we construct an implicit Nonstandard Finite Difference (NSFD) scheme that, under an  existence and uniqueness condition, is an explicit time reversible scheme. Apart from being elementary stable, we show that the  NSFD scheme is of second-order and domain-preserving, thereby solving a pending problem on the construction of higher-order nonstandard schemes without spurious solutions, and extending the tangent condition to discrete dynamical systems. It is shown that the new scheme applies directly for mass action-based models of biological and chemical processes.
\end{abstract}

\section{Introduction}

The general setting of this work is a dynamical system, on the set $\Omega\subset\mathbb{R}^n$,  defined by the following system of autonomous ordinary differential equations:
\begin{equation}\label{ODEs}
\frac{dx}{dt}\equiv \dot{x}=f(x).
\end{equation}
The solutions, $x(t)$, of the continuous system are approximated  at the discrete times, $t_k=hk$, $h>0$, $k=0,1,2, \cdots$, by a sequence $(x_k)$ obtained recursively through a numerical method of the following form:
\begin{equation}\label{NumScheme}
x_{k+1}=F(h,x_k)\equiv F(h)(x_k).
\end{equation}
For both the continuous and the discrete systems, it is assumed that the functions $f\equiv f(x):\mathbb{R}^{n}\rightarrow\mathbb{R}^n$ and $F\equiv F(h,x)\equiv F(h)(x):\mathbb{R}\times\mathbb{R}^{n}\rightarrow\mathbb{R}^n$ are as smooth as needed. Naturally,  the consistency requirements,
\begin{equation}\label{consistency}
F(0)(x)=x \mbox{ and } \frac{dF(0)}{dh}(x)=f(x),
\end{equation}
are assumed to hold, and thus we suppose that  $F$ is at least continuously differentiable on $h$.
Our aim is that the numerical method (\ref{NumScheme}) be reliable in the three important directions below.
The numerical method must be second-order convergent, elementary stable, and
preserve the property of having the set $\Omega$ forward invariant with respect to the continuous system.

The construction of higher-order NSFD schemes that are dynamically consistent with the underlying features of the continuous differential equations models, particularly for those with  transient dynamics, is a pending problem. Several authors have attempted to address this problem. These include the nonstandard versions of the classical $\theta$-method, with $\theta=1/2$,  developed in
\cite{Anguelov_ber,Anguelov_Dumont,Anguelov:2005,Lub-Roux:2003} where the positivity of the discrete solution is achieved by also using Mickens' rule on a complex denominator function for the discrete derivative. In the same vein, we mention the recent work \cite{Hristo}, valid for a scalar differential equation, where the second-order accuracy and elementary stability are achieved by a modified nonstandard $\theta$ method, $0\leq\theta\leq 1$, in which the complex denominator function varies at each iteration. \\

\noindent In this paper, we deal with numerical methods characterized as follows \cite[Section V.1]{Hairer2006}:
\begin{definition}\label{def}
A numerical scheme (\ref{NumScheme}) is called symmetric or time-reversible if
\begin{equation}\label{reverse}
x_k=F(-h,x_{k+1}).
\end{equation}
More generally, the scheme is said to be reversible for all $x$ whenever $y=F(h,x)$ implies that $x=F(-h,y)$.
\end{definition}
We construct a NSFD scheme that is time-reversible and show that, in essence, this property makes our scheme of second-order accuracy and elementary stable (Section 2). The same property is used to establish a discrete analogue of the tangent condition for forward invariance of a convex set, which implies that the domain under consideration is positively invariant with respect to our NSFD scheme (Section 3). The case of mass action models of biological and chemical processes is considered in details with Kermack-McKendrik  Susceptible-Infective (\textit{SI}) epidemic model \cite{Heth00}, as an illustrative example (Section 4). Concluding remarks, with possible extensions of this work,  are given in Section~5
\section{Nonstandard finite difference method of order 2}

In the construction of NSFD schemes, Mickens \cite{Mickens:1994} made an important observation namely, that the discrete models of differential equations have a larger parameter space than the corresponding differential equations. Adding to this Mickens' Rule 1, which says that the orders of the discrete derivatives must be exactly equal to the orders of the corresponding derivatives of the differential equations, it  is convenient for nonlocal approximations to view the right-hand side of the system (\ref{ODEs}) as a restriction, on the diagonal $z=y$, of a certain function of two variables $\varphi\equiv\varphi(y,z)$ such that $f(x)=\varphi(x,x)$ and hence the system takes the form
\begin{equation}\label{CDSNonLoc}
\dot{x}=\varphi(x,x), \ \ x\in\Omega\subseteq\mathbb{R}^n.
\end{equation}
We consider the numerical method,
\begin{equation}\label{NonLocGen}
\frac{x_{k+1}-x_k}{h}=\frac{1}{2}(\varphi(x_{k+1},x_k)+\varphi(x_k,x_{k+1})),
\end{equation}
which is a NSFD scheme \cite{AngLub00}, on which only the nonlocal approximation of the right-hand side of (\ref{CDSNonLoc}) is used, the rule on the complex denominator function of the derivatives being excluded. The NSFD scheme (\ref{NonLocGen}) is implicit. The existence and uniqueness of solution of the respective system of algebraic equation is an issue which requires attention on its own. To make this NSFD scheme a generalized dynamical system on $\Omega$ \cite{StuHum98}, we will assume here that
\begin{equation}\label{AssumptionUniqueSol}
\begin{tabular}{l}
(i) given $x_k\in\Omega$, equation (\ref{NonLocGen}) has a unique solution $x_{k+1}$ in $\Omega$;\\
(ii) different values of  $x_k$ result in different solutions for $x_{k+1}$.
\end{tabular}
\end{equation}
Finding general conditions for (\ref{AssumptionUniqueSol}) to hold is challenging, as it is the case for general equations of the form $G(x,y)=0$ solved by using for instance the implicit function theorem, sophisticated fixed-point iterations, \textit{etc.} However, in particular settings, as the ones considered in the sequel such conditions are relatively easy to formulate.

\begin{theorem}\label{theo1}
Assume that (\ref{AssumptionUniqueSol}) holds. Then the
\begin{itemize}
\item[a)] The NSFD scheme (\ref{NonLocGen}) is time-reversible;
\item[b)] The fixed points of the NSFD scheme (\ref{NonLocGen}) are precisely the equilibrium points of the system (\ref{CDSNonLoc});
\item[c)] The NSFD scheme (\ref{NonLocGen}) replicates the local stability of all hyperbolic equilbrium points of (\ref{CDSNonLoc});
\item[d)] The NSFD scheme (\ref{NonLocGen}) is a second-order scheme for the equation (\ref{CDSNonLoc}).
\end{itemize}
\end{theorem}
\begin{proof}
The assumption (\ref{AssumptionUniqueSol}) implies that the NSFD scheme (\ref{NonLocGen}) can be written in the explicit form  (\ref{NumScheme}) where where the function $F$ satisfies the equation
\begin{equation}\label{EqF}
F(h,x)-x=\frac{h}{2}\left[\varphi(F\left(h,x),x\right)+\varphi\left(x,F(h,x)\right)\right].
\end{equation}
\noindent a) If in the equation (\ref{NonLocGen}) we interchange $x_k$ and $x_{k+1}$ and replace $h$ by $-h$, the equation remains the same. Therefore (\ref{reverse}) holds, which means that the NSFD scheme is time-reversible.\\

\noindent b) It is clear that $\overline{x}$ is a fixed point of (\ref{NonLocGen}) if, and only if, $\varphi(\overline{x},\overline{x})=0$, \textit{i.e.}  $\overline{x}$ is an equilibrium point of (\ref{CDSNonLoc}).\\

\noindent c) Let us use the explicit form (\ref{NumScheme}) of (\ref{NonLocGen}).
From (\ref{EqF}), the Jacobian matrix (in the $x$ variable) of $F$ satisfies
\begin{eqnarray}
\frac{\partial F(h,x)}{\partial x}-I&=&\frac{h}{2}\left(
\frac{\partial\varphi}{\partial y}(F(h,x),x)\frac{\partial F(h,x)}{\partial x}+\frac{\partial\varphi}{\partial z}(F(h,x),x)\right.\nonumber\\
&&\ \ \ \ \left. +\frac{\partial\varphi}{\partial y}(x,F(h,x))+\frac{\partial\varphi}{\partial z}(x,F(h,x))\frac{\partial F(h,x)}{\partial x}\right)\label{EqF2}
\end{eqnarray}
Let $\overline{x}$ be an arbitrary fixed point of (\ref{NonLocGen}). At $x=\overline{x}$, using $F(h,\overline{x})=\overline{x}$, equation (\ref{EqF2}) simplifies to
\[
\left(I-\frac{h}{2}\left(\frac{\partial\varphi}{\partial y}(\overline{x},\overline{x})+\frac{\partial\varphi}{\partial z}(\overline{x},\overline{x})\right)\right)\frac{\partial F(h,\overline{x})}{\partial x}=
I+\frac{h}{2}\left(\frac{\partial\varphi}{\partial y}(\overline{x},\overline{x})+\frac{\partial\varphi}{\partial z}(\overline{x},\overline{x})\right),
\]
or. equivalently,
\begin{equation}\label{eqJac}
\left(I-\frac{h}{2}\frac{\partial f(\overline{x})}{\partial x}\right)\frac{\partial F(h,\overline{x})}{\partial x}=
I+\frac{h}{2}\frac{\partial f(\overline{x})}{\partial x}.
\end{equation}

Let $\lambda\equiv \Re\lambda +\imath\Im\lambda$ be an eigenvalue of $\displaystyle\frac{\partial f(\overline{x})}{\partial x}$ with associated left eigenvector $v$. Then multiplying both sides of \eqref{eqJac}    on the left by $v$ we obtain
\[
\left(1-\frac{h}{2}\lambda\right)v\,\frac{\partial F(h,\overline{x})}{\partial x} = \left(1+\frac{h}{2}\lambda\right)v
\]
It follows from \eqref{AssumptionUniqueSol} that the matrix $\displaystyle\frac{\partial F(h,\overline {x})}{\partial x}$ is not singular. Then, due to the obvious impossibility for $1-\frac{h}{2}\lambda$ and $1+\frac{h}{2}\lambda$ to be both zero, neither of them is. Hence
\[
v\,\frac{\partial F(h,\overline{x})}{\partial x} =\frac{1+\frac{h}{2}\lambda}{1-\frac{h}{2}\lambda}\, v.
\]
Therefore, to every eigenvalue $\lambda$ of $\displaystyle\frac{\partial f(\overline{x})}{\partial x}$ with associated left eigenvector $v$, there corresponds an eigenvalue $\displaystyle\mu=\frac{1+\frac{h}{2}\lambda}{1-\frac{h}{2}\lambda}$ of the matrix $\displaystyle\frac{\partial F(h,\overline{x})}{\partial x}$ with the same left eigenvector $v$. After some technical manipulation we obtain
\begin{eqnarray}
|\mu|^2&=&\frac{(1+\frac{h}{2}\Re(\lambda))^2+\frac{h^2}{4}\Im^2(\lambda)}{(1-\frac{h}{2}\Re(\lambda))^2+\frac{h^2}{4}\Im^2(\lambda)}\nonumber\\
&=&\frac{1+\frac{h^2}{4}|\lambda|^2+h\Re(\lambda)}{1+\frac{h^2}{4}|\lambda|^2-h \Re(\lambda)}\label{mulambda}
\end{eqnarray}
If follows from (\ref{mulambda}) that
\[
\Re(\lambda)<0\ \ \Longleftrightarrow \ \ |\mu|<1,
\]
which implies that a hyperbolic equilibrium point $\overline{x}$ of (\ref{CDSNonLoc}) is asymptotically stable if, and only if, it is asymptotically stable as a fixed-point of (\ref{NonLocGen}).\\

\noindent d)  Let $x\in\Omega$ and let $u$ denote the solution of (\ref{ODEs}) with $u(0)=x$. Denote
\[
\xi=\frac{1}{2}(F(h)(x)+x) .
\]
Using the consistency conditions (\ref{consistency}) one can show that
\begin{equation}\label{eq11}
\xi=\frac{1}{2}\left(x+F(0,x)+\frac{dF}{dh}(0,x)+O(h^2)\right)=x+\frac{h}{2}f(x)+O(h^2)=u\left(\frac{h}{2}\right)+O(h^2)
\end{equation}
Then, for  the local truncation error,
\[
E(h)=u(h)-F(h,x)=u(h)-x-\frac{h}{2}\varphi(F(h)(x),x)-\frac{h}{2}\varphi(x,F(h)(x)),
\]
of the NSFD scheme (\ref{NumScheme}) \& (\ref{EqF}), Taylor expansions of $u$ about $\frac{h}{2}$ and of $\varphi$ about $(y,z)=(\xi,\xi)$ yield
\begin{eqnarray*}
E(h)&=&h\dot{u}\left(\frac{h}{2}\right)-\frac{h}{2}\left(\varphi(\xi,\xi)+\frac{\partial \varphi(\xi,\xi)}{\partial y}(F(h)(x)-\xi)+\frac{\partial \varphi(\xi,\xi)}{\partial z}(x-\xi)\right)\\
&&-\frac{h}{2}\left(\varphi(\xi,\xi)+\frac{\partial \varphi(\xi,\xi)}{\partial y}(x-\xi)+\frac{\partial \varphi(\xi,\xi)}{\partial z}(F(h)(x)-\xi)\right)+O(h^3)\\
&=&h\left(f\left(u\left(\frac{h}{2}\right)\right)-f(\xi)\right)+O(h^3).
\end{eqnarray*}
Then using (\ref{eq11}) we obtain $E(h)=O(h^3)$ and second order accuracy follows from the standard theory for one-step numerical methods.
\end{proof}

\begin{remark} Since the scheme (\ref{NonLocGen}) is symmetric, the statement d) of the Theorem can be derived from the general theory of symmetric schemes \cite[Theorem II.3.2]{Hairer2006}.
\end{remark}

\section{Domain preserving property of reversible schemes}

The theory for continuous dynamical systems of the form (\ref{ODEs}) provides theorems proving that a set is forward invariant through conditions only on the boundary of the set, e.g. \cite[Section 10]{Walter}. In this section we derive similar theorems for reversible maps, that is appropriate assumptions on the boundary of a set imply that the set is forward invariant.

Let $D$ be a closed subset of $\Omega$. We denote by $\overset{\circ}{D}$ the interior of $D$, by $\partial D$ - the boundary of $D$ and assume that
\begin{equation}\label{closureD}
D=closure(\overset{\circ}{D})=\overset{\circ}{D}\cup\partial D.
\end{equation}

\begin{theorem}\label{theoDomain1}
Let a scheme (\ref{NumScheme}) be reversible for all
$h\in(0,\overline{h}]$. If
\begin{equation}\label{cond1}
F(-h,\partial D)\cap\overset{\circ}{D}=\emptyset,\ h\in (0,\overline{h}],
\end{equation}
then the set $D$ is invariant under $F(h,\cdot)$ for every $h\in(0,\overline{h})$.
\end{theorem}
\begin{proof}
Let $x\in\overset{\circ}{D}$. Assume that there exists $h\in(0,\overline{h})$ such that $F(h,x)\notin \overset{\circ}{D}$. Since $F(h,x)$ is continuous on $h$, there exists $\hat{h}\in(0,h]$ such that $y=F(\hat{h},x)\in\partial D$. Equivalently, $x=F(-\hat{h},y)$, which contradicts the condition (\ref{cond1}). Therefore, $F(h,x)\in \overset{\circ}{D}$. Considering that $x$ is an arbitrary element of $\overset{\circ}{D}$, we have $F(h,\overset{\circ}{D})\subseteq \overset{\circ}{D}$. Using the continuity of $F$ on $x$ and (\ref{closureD}) we obtain $F(h,D)\subseteq D$.
\end{proof}\\

Condition (\ref{cond1}) is not easy to verify. It can be replaced by a stronger condition, requiring that the vector $F(-h,x)-x$ be outward directed or tangential to $D$ at $x\in\partial D$. This condition is similar to the tangent condition in  \cite[Section 10.XV]{Walter}, which is the essential property for invariance of sets under flows. In the setting of the system \eqref{ODEs} for the set $D$ the tangential condition states that
\begin{equation}\label{tangentC}
n(x)\cdot f(x)\leq 0,
\end{equation}
for every point $x\in\partial D$ and any outer normal vector $n(x)$ at $x$. Let us recall that a vector $n(x)$ is called outer normal vector to $D$ at $x$ if the circle with center $x+n(x)$ and radius $|n(x)|$ has no intersection with $\overset{\circ}{D}$. The straight line through $x$ perpendicular to $n(x)$ is called a tangent. This general definition of normal vector and tangent does not require any smoothness of the boundary $\partial D$, which is rather convenient in applications. Let us note that the normal vector $n(x)$, similarly the associated tangent, need not exist and, if existing, need not be unique.

In the next theorem we make an additional assumption on $D$ to be convex. Convexity of $D$ implies that a normal vector and a tangent exist at every point $x\in\partial D$. Further, a defining property of $D$ being convex is that $\overset{\circ}{D}$ is on one side of any tangent at any point $x\in\partial D$. Specifically, on the side not containing $x+n(x)$.

\begin{theorem}\label{theoDomain}
Let a scheme (\ref{NumScheme}) be reversible for all
$h\in(0,\overline{h}]$ and let $D$ be a closed and convex subset of $\Omega$ satisfying \eqref{closureD}.  If for every $h\in(0,\overline{h}]$, $x\in\partial D$ and an outer normal vector $n(x)$ we have
\begin{equation}\label{tangentC2}
n(x)\cdot (F(-h,x)-x)\geq 0,
\end{equation}
then the set $D$ is invariant under $F(h,\cdot)$ for every $h\in(0,\overline{h})$.
\end{theorem}
\begin{proof}
Let $h\in(0,\overline{h}]$ and $x\partial D$. The inequality \eqref{tangentC2} implies that the point $x+(F(-h,x)-x)$ is on the same side as $x+n(x)$ with respect to the tangent to $D$  at $x$ and associated with $n(x)$. Considering that $D$ is convex, this implies that $F(-h,x)\notin \overset{\circ}{D}$. Then the conclusion follows from Theorem \ref{theoDomain1}.
\end{proof}\\





\section{Schemes for mass action models}

Modeling of biological and chemical processes by applying the mass action principle for representing the interaction of the involved species typically results in a system of the form (\ref{CDSNonLoc}) in $\Omega$, a convex and compact subset of  $\mathbb{R}^n_+$, where the function $\varphi$ is linear in both its arguments. Then $\varphi$ can be represented as
\begin{equation}\label{varphiMA}
\varphi(x,y)=P(x)y+A(x+y)+b=Q(y)x+A(x+y)+b
\end{equation}
where $P$ and $Q$ are $n\times n$ matrix functions of $x$ and $y$, respectively. The matrix $A$ and the vector $b$ are constant.
In this particular setting the numerical method (\ref{NonLocGen}) can be written in the form
\begin{equation}\label{NonLocMA1}
\frac{x_{k+1}-x_k}{h}=\frac{1}{2}(P(x_k)x_{k+1}+Q(x_k)x_{k+1})+A(x_{k+1}+x_k)+b
\end{equation}
or, equivalently,
\begin{equation}\label{NonLocMA2}
\left(I-\frac{h}{2}\left(P(x_k)+Q(x_k)\right)-hA\right)x_{k+1}=(I+hA)x_k+hb,
\end{equation}
where $I$ is the $n\times n$ identity matrix. Furthermore, and also expressing it reversibility, the equation  (\ref{NonLocMA1}) can be written in the form
\begin{equation}\label{NonLocMA3}
\left(I+\frac{h}{2}\left(P(x_{k+1})+Q(x_{k+1})\right)+hA\right)x_{k}=(I-hA)x_{k+1}-hb.
\end{equation}

Considering that $\Omega$ is compact, there exists $\overline{h}$ such that
\begin{equation}\label{diagDom}
\begin{aligned}
&I-\frac{h}{2}(P(x)+Q(x))-hA \ \text{ and }\  I+\frac{h}{2}(P(x)+Q(x))+hA \\
&\text{are both (row or column) diagonally dominant matrices}\\
&\text{ for all } x\in(0,\overline{h}] \text{ and }x\in\Omega.
\end{aligned}
\end{equation}

 Then, these matrices are invertible and (\ref{NonLocMA2}) can be written explicitly as
\begin{equation}\label{NonLocMA1Exp}
x_{k+1}=F(h,x_k)
\end{equation}
where
\begin{equation}\label{FMA}
F(h,x)=\left(I-\frac{h}{2}(P(x)+Q(x))-hA\right)^{-1}((I+hA)x+hb).
\end{equation}
Further, we have
\begin{equation}\label{FinvMA}
F^{-1}(h,x)=F(-h,x)=\left(I+\frac{h}{2}(P(x)+Q(x))+hA\right)^{-1}((I-hA)x-hb).
\end{equation}

\begin{theorem}\label{theoMA}
Let there exist $\bar{h}>0$ such that \eqref{diagDom} holds and $\Omega$ is an invariant set of the map $F(h,\cdot)$ for all $h\in(0,\overline{h}]$. Then the scheme \eqref{NonLocMA1} satisfies the properties a)--d) in Theorem \ref{theo1}.
\end{theorem}
\begin{proof}
The scheme \eqref{NonLocMA1} is constructed in the general form \eqref{NonLocGen}. Therefore, it is sufficient to show that condition \eqref{AssumptionUniqueSol} holds. The solution is given in the explicit form \eqref{NonLocMA1Exp}. Therefore, it exists for every $h\in(0,\overline{h}]$ and it is unique. Further, it belongs to $\Omega$, due to the assumption that $\Omega$ is invariant.
\end{proof}

\begin{remark}
The existence of $\bar{h}$ such \eqref{diagDom} holds follows from the compactness of $\Omega$. In practice, it is determined from the condition of diagonal dominance using the boundedness of the variable $x$.
\end{remark}
\begin{remark}
Theorem \ref{theoDomain} is a useful tool in determining the the invariance of $\Omega$ under $F(h,\cdot)$.
\end{remark}
\begin{remark}
It very common in applications that $b\geq 0$ and the nondiagonal entries of $A$, $P(x)$, $Q(x)$, $x\in\Omega$ are nonnegative, that is, these matrices are Metzler matrices. Then, assuming that \eqref{diagDom} holds,
\[
I-\frac{h}{2}(P(x)+Q(x))-hA
\]
is an M-matrix. Hence, its inverse is a nonnegative matrix, \cite[Theorem 2.3]{BermonPlemmons1994}. Therefore, it follows from \eqref{NonLocMA1Exp} that
\begin{equation}\label{positivityProp}
\Omega\subseteq \mathbb{R}^n_+\ \Longrightarrow \
F(h,\Omega)\subseteq \mathbb{R}^n_+, \ h\in(0,\bar{h}].
\end{equation}
The property \eqref{positivityProp} is usually a first step towards proving the invariance of $\Omega$ under $F(h,\cdot)$, $h\in(0,\overline{h}]$.
\end{remark}


\section{Application to a model in Mathematical Epidemiology}

In this section we show how the theory and tools developed in Sections 2, 3 and 4 can be put together in deriving  second order scheme with the properties a)--c) in Theorem \ref{theo1}. We note that properties a)--c) are of qualitative nature and do not follows from the standard numerical analysis theory involving stability, consistency and therefore convergence. While convergence as $h\to 0$ is expected, the computations are always done for some positive value of $h$. Hence, the importance that the main characteristics of the dynamical system are preserved as given in properties a)--c) for $h\in(0,\overline{h}]$. While the theoretical existence of $\overline{h}$ is important, deriving constructively value of $\overline{h}$ is critical for any specific application. We show how the value of $\overline{h}$  can be computed for the considered model. It should be noted that, while the numerical approach proposed in this paper is exemplified on the model considered in this section, its realm of application is not limited to it.

We consider the model of vector borne diseases with temporal immunity as given in \cite[Section 4.4, (4.39)--(4.40)]{Martcheva}
\begin{eqnarray}
\frac{dS_v}{dt}&=&\Lambda_v-pS_v I-\mu_v S_v,\label{Mart1}\\
\frac{dI_v}{dt}&=&pS_vI-\mu_vI_v,\label{Mart2}\\
\frac{dS}{dt}&=&\Lambda-qSI_v-\mu S+\gamma R,\label{Mart3}\\
\frac{dI}{dt}&=&qSI_v-(\mu+\alpha)I,\label{Mart4}\\
\frac{dR}{dt}&=&\alpha I-(\mu+\gamma) R,\label{Mart5}
\end{eqnarray}
where $S_v$ and $I_v$ are the susceptible and infective vectors and $S$, $I$, $R$ are the susceptible, invective and recovered (with temporary immunity) hosts, e.g. humans. All parameters are positive. Bearing in mind the units $\Lambda_v$, $\Lambda$ are the recruitments, $\mu_v$, $\mu$ - the death rates for vectors and hosts, respectively, $p$ and $q$ are the probabilities of sufficient contact for transmission from host to vector and from vector to host, the average period of infectivity of host is $\alpha^{-1}$, while the average duration of immunity of host is $\gamma^{-1}$. Let us note that an additional parameter $a$ in \cite{Martcheva} is absorbed here in the parameters $p$ and $q$.

Let $x=(S_v,I_v,S,I,R)^T$. Then the system \eqref{Mart1}--\eqref{Mart5} can be stated in the form \eqref{CDSNonLoc} with $\varphi$ given by \eqref{varphiMA}, where
\begin{equation}\label{ExPQAb}
\begin{aligned}
&P(x)=\left(\begin{matrix}-pI&0&0&0&0\\
pI&0&0&0&0\\
0&0&-qI_v&0&0\\
0&0&qI_v&0&0\\
0&0&0&0&0
\end{matrix}\right),\ \
Q(x)=\left(\begin{matrix}0&0&0&-pS_v&0\\
0&0&0&pS_v&0\\
0&-qS&0&0&0\\
0&qS&0&0&0\\
0&0&0&0&0
\end{matrix}\right),\\[6pt]
&
A=\left(\begin{matrix}-\mu_v&0&0&0&0\\
0&-\mu_v&0&0&0\\
0&0&-\mu&0&\gamma\\
0&0&0&-(\mu+\alpha)&0\\
0&0&0&\alpha&-(\mu+\gamma)
\end{matrix}\right),\ \
b=\left(\begin{matrix}\Lambda_v\\0\\\Lambda\\0\\0
\end{matrix}\right).
\end{aligned}
\end{equation}

The set of equations \eqref{Mart1}--\eqref{Mart5} defines a dynamical system on
\[
\Omega=\{x\in\mathbb{R}^5_+:S_v+I_v\leq M_v,S+I+R\leq M\},
\]
where $M_v$ and $M$ are real numbers such that
\begin{equation}\label{lowBoundsM}
M_v\geq \frac{\Lambda_v}{\mu_v},\ \ M\geq\frac{\Lambda}{\mu}.
\end{equation}
It is easy so see, and it is also shown in \cite{Martcheva}, that for the populations $N_v=S_v+I_v$, $N=S+I+R$ of vectors and hosts, respectively, we have
\begin{equation}\label{limitsN}
\lim_{t\to+\infty}N_v(t)=\frac{\Lambda_v}{\mu_v},\ \ \lim_{t\to+\infty}N=\frac{\Lambda}{\mu}.
\end{equation}
Hence, the inequalities \eqref{lowBoundsM} are important. The limit properties \eqref{limitsN} are used in \cite{Martcheva} to reduce the asymptotic analysis of \eqref{Mart1}--\eqref{Mart5} to asymptotic analysis of a  system to two equations. We consider here the full system \eqref{Mart1}--\eqref{Mart5}, since the goal is to compute accurately the trajectories for all finite times and not only their limit at $t\to+\infty$. Further, this illustrates that the applicability of the proposed numerical approach is not restricted in any way by the dimensionality of the system.

The vector field $f(x)=\phi(x,x)$, defined by right hand side of \eqref{Mart1}--\eqref{Mart5} satisfies the tangential condition in \cite[Section 10.XV]{Walter}. Specifically, on the planes
\begin{eqnarray}
 &&S_v+I_v=M_v,\label{plane1}\\
 &&S+I+R=M,\label{plane2}
\end{eqnarray}
with respective outward normal vectors $u=(1,1,0,0,0)^T$ and $w=(0,0,1,1,1)^T$, we have
\begin{eqnarray}
u\cdot f(x)&=&\Lambda_v-\mu_v (S_v+I_v)=\Lambda_v-\mu M_v\leq 0,\label{ufleq0}\\
w\cdot f(x)&=&\Lambda-\mu(S+I+R)=\Lambda-\mu M\leq 0.\label{wfleq0}
\end{eqnarray}

We consider the numerical method \eqref{NonLocMA1Exp} for the system \eqref{Mart1}--\eqref{Mart5} on $\Omega$. Our aim is apply Theorem \ref{theoMA}. From the structure of the matrices  $P$, $Q$ and $A$ it follows that the matrix $\displaystyle I-\frac{h}{2}(P(x)+Q(x))-hA$ is column diagonally dominant for every $h> 0$ and $x\geq 0$. The column diagonal dominance of the matrix $\displaystyle I+\frac{h}{2}(P(x)+Q(x))+hA$ is equivalent to the following set of inequalities:
\begin{equation}\label{ineq}
\begin{aligned}
h(pI+\mu_v)&\leq&1,\\
h(qS+\mu_v&\leq&1,\\
h(qI_v+\mu)&\leq&1,\\
h(pS_v+\mu+\alpha)&\leq&1,\\
h(\mu+2\gamma)&\leq&1.
\end{aligned}
\end{equation}
Therefore, condition \eqref{diagDom} holds for all $x\in\Omega$ and $h\in(0,\overline{h}]$, where
\begin{equation}\label{bar_hMA}
\overline{h}=\min\left\{\frac{1}{pM+\mu_v},\frac{1}{qM+\mu_v}, \frac{1}{qM_v+\mu},\frac{1}{pM_v+\mu+\alpha},\frac{1}{\mu+2\gamma}\right\}.
\end{equation}

It remains to show the invariance of $\Omega$. Considering that $P$, $Q$, $A$ are Metzler matrices, we have that \eqref{positivityProp} holds. Therefore, it is sufficient to apply Theorem \ref{theoDomain} on the part of the boundary defined by planes \eqref{plane1}--\eqref{plane2}. Using \eqref{FinvMA}, we have
\begin{eqnarray}
&&\hspace{-1cm}\left(I+\frac{h}{2}(P(x)+Q(x))+hA\right)(F(-h,x)-x)\nonumber\\
&&=(I-hA)x-hb-(I+\frac{h}{2}(P(x)+Q(x))+hA)x\nonumber\\
&&=\frac{h}{2}(P(x)x+2Ax)+\frac{h}{2}(Q(x)x+2Ax\nonumber\\
&&=-\frac{h}{2}\varphi(x,x)-\frac{h}{2}\varphi(x,x)\ =\ -hf(x)\label{aux}
\end{eqnarray}

It is easy to see from the structure of $P$, $Q$ and $A$ in \eqref{ExPQAb} that
\begin{equation}\label{aux2}
u^TP=w^TP=u^TQ=w^TQ=(0,0,0,0,0)^T,\ \ u^TA=-\mu_vu,\ \ w^TA=-\mu w.
\end{equation}

Putting together \eqref{aux} and \eqref{aux2} we obtain
\begin{eqnarray*}
(1-h\mu_v)u\cdot(F(-h,x)-x)&=&u^T\left(I+\frac{h}{2}(P(x)+Q(x))+hA\right)(F(-h,x)-x)\\
&=&-hu\cdot f(x)\geq 0.
\end{eqnarray*}
Using that $h\leq\overline{h}<\frac{1}{\mu_v}$ and \eqref{ufleq0}, we have
\[
u\cdot(F(-h,x)-x)=-\frac{h}{1-h\mu_v}u\cdot f(x)\geq 0
\]
for all $x$ on the plane \eqref{plane1} and $h\in(0,\overline{h}]$. Similarly, we obtain
\[
w\cdot (F(-h,x)-x)=-\frac{h}{1-h\mu_v}u\cdot f(x)\geq 0
\]
for all $x$ on the plane \eqref{plane2} and $h\in(0,\overline{h}]$. Then it follows from Theorem \ref{theoDomain} that $\Omega$ is positively invariant under all maps $F(h,\cdot)$, $h\in(0,\overline{h}]$. Hence the numerical method \eqref{NonLocMA1Exp} for the system \eqref{Mart1}--\eqref{Mart5} with $\overline{h}$ given in \eqref{bar_hMA} satisfies properties a)--d) in Theorem \ref{theo1}.

\section{Conclusion}
The literature on high-order numerical methods for dynamical systems defined by ordinary differential equations is rich. However, these classical schemes can exhibit spurious/ghost solutions or other elementary instability that do not correspond to the feature of the continuous equations. We overcome this difficulty by constructing a time-reversible nonstandard finite difference (NSFD) scheme, with the classical denominator $h\equiv\Delta t$ for the discrete derivative but a nonlocal discretization of the right-hand side of the differential equations. Our findings, directly applicable to mass action-type models in biology, are twofold. First, the NSFD scheme is second-order convergent and elementary stable. Second, we formulate a discrete analogue of the tangent condition under which it is shown that the NSFD scheme preserves the propriety of the continuous model stating that its feasible region is forward invariant.

Our future research includes  extending this study to a dynamical system with nonhyperbolic equilibrium points, an approach considered in \cite{AngLubShillor2011} as far as the stability for one-dimensional model is concerned.


\begin{thebibliography}{99}
\bibitem{Anguelov_ber}R. Anguelov,  T. Berge, M. Chapwanya, J.K. Djoko, P. Kama,  J. M-S. Lubuma and Y. Terefe, Nonstandard Finite difference method revisited and application to the  Ebola Virus Disease dynamics transmission, \textit{Journal of Difference Equations and Applications} \textbf{26}, 2020(6), 818-854.
\bibitem{Anguelov_Dumont}R. Anguelov, Y. Dumont, J.M-S. Lubuma and M. Shillor, Dynamically consistent nonstandard finite difference schemes for epidemiological models, \textit{Journal of Computational  and Applied Mathematics,} \textbf{255}(2014), 161-182.
\bibitem{Anguelov:2005} R. Anguelov, P. Kama and J.M.-S. Lubuma, On
nonstandard finite difference models of reaction-diffusion equations,
\textit{J. Computational and Applied Mathematics} \textbf{175} (2005),
11--29.
\bibitem{AngLub00} R. Anguelov and J.M-S. Lubuma, Contributions to the
mathematics of the nonstandard finite difference method and applications,
\textit{Numerical Methods for Partial Differential Equations} \textbf{17}
(2001), 518--543.

\bibitem{AngLub03} R. Anguelov, J.M-S. Lubuma, Nonstandard Finite Difference
Method by Nonlocal Approximation, \textit{Mathematics and Computers in
Simulation} \textbf{61}(3-6) (2003), 465--475.


\bibitem{AngLubShillor2011} R Anguelov, J M-S Lubuma, M. Shillor, Topological dynamic consistency of nonstandard finite difference schemes for dynamical systems, Journal of Difference Equations and Applications, 17 (12), (2011), 1769-1791

\bibitem{BermonPlemmons1994} A Bermon, R J Plemmons, Nonnegative matrices in the mathematical sciences, SIAM, 1994.





\bibitem{Dimitrov06} D. T. Dimitrov and H.V. Kojouharov, Positive and
elementary stable nonstandard numerical methods with applications to
predator-prey models, \textit{Journal of Computational and Applied
Mathematics} \textbf{189} (2006), 98--108.







\bibitem{Hairer2006} E. Hairer, C. Lubich, G. Wanner, Geometric Numerical Integration: Structure-Preserving Algorithms for Ordinary Differential Equations, Springer, Berlin, Heidelberg, New York, 2006.

\bibitem{Heth00} H.W. Hethcote, The mathematics of infectious disease, \textit{SIAM Rev.} \textbf{42} (2000), 599--653.


\bibitem{Hristo}H.V. Kojouharov, S. Roy, M. Gupta, F. Alalhareth and J.M. Slezak, A second-order modified nonstandard theta method for autonomous differential equations, \textit{Applied Mathematics Letters} \textbf{112}, (2021), 106775.






\bibitem{Lub-Roux:2003} J.M-S. Lubuma and A. Roux, An improved theta method for systems of ordinary differential equations, \textit{J. Difference
Equations and Applications}, \textbf{9} (2003), 1023--1035.

\bibitem{Martcheva} M Martcheva, Introduction to Mathematical Epidemiology, Springer, 2015.




\bibitem{Mickens:1994} R.E. Mickens, Nonstandard finite difference models of differential equations, World Scientific, Singapore, 1994.


\bibitem{Mickens:2000} R.E. Mickens (Ed.), Applications of Nonstandard Finite Difference Schemes, World Scientific, Singapore, 2000.


\bibitem{Mickens:2005} R.E. Mickens, Nonstandard finite difference methods, In: R.E. Mickens (Ed), \textit{Advances in the applications of nonstandard finite difference schemes}, World Scientific, Singapore, 2005, pp. 1--9.

\bibitem{Mickens2005JDEA} R.E. Mickens, Dynamic consistency: a fundamental principle for constructing nonstandard finite difference schemes for differential equations, Journal of Difference Equations and Applications 11 (2005) 645-653.





\bibitem{StuHum98} A.M. Stuart and A.R. Humphries, \textit{Dynamical systems and numerical analysis}, Cambridge University Press, New York, 1998.
\bibitem{Walter}W. Walter, \textit{Ordinary differential equations}, New York, Springer, 1998.
\end{thebibliography}
\end{document}